\documentclass[12pt]{article}
\usepackage{amsthm}
\usepackage{amsfonts}
\usepackage{epsfig}
\usepackage{amssymb}
\usepackage{amsmath}
\usepackage{amscd}
\usepackage{colonequals}
\usepackage{url}
\newtheorem{theorem}{Theorem}
\newtheorem{lemma}[theorem]{Lemma}
\newtheorem{corollary}[theorem]{Corollary}

\newcommand{\GG}{\mathcal{G}}
\newcommand{\PP}{\mathcal{P}}
\newcommand{\bb}[1]{\mathbf{#1}}

\newcommand{\dd}{\widetilde{\nabla}}
\newcommand{\ddin}{\widetilde{\nabla}^\circ}
\newcommand{\ddinex}{\widetilde{\nabla}^e}
\title{Induced subdivisions and bounded expansion}

\author{Zden\v{e}k Dvo\v{r}\'{a}k\thanks{Charles University, Prague, Czech Republic.
E-mail: {\tt rakdver@iuuk.mff.cuni.cz}.  Supported by the Center of Excellence -- Inst. for Theor. Comp. Sci., Prague (project P202/12/G061 of Czech Science Foundation).}}
\date{}
\begin{document}
\maketitle

\begin{abstract}
We prove that for every graph $H$ and for every integer $s$, the class of graphs
that do not contain $K_s$, $K_{s,s}$, or any subdivision of $H$ as an induced subgraph
has bounded expansion; this strengthens a result of K\"uhn and Osthus~\cite{kuhndens}.
The argument also gives another characterization of graph classes with bounded expansion
and of nowhere-dense graph classes.
\end{abstract}

For a non-negative integer $k$, a \emph{$(\le\!k)$-subdivision} of a graph $H$ is a graph obtained from $H$ by subdividing
each of its edges by at most $k$ vertices (not necessarily the same number on each edge).
For a graph $G$, let $\dd_k(G)$ denote the maximum average degree of a graph $H$ whose $(\le\!k)$-subdivision appears as a subgraph in $G$;
in particular, $\dd_0(G)$ is the maximum average degree of a subgraph of $G$.  We say that a class of graphs $\GG$ has \emph{bounded expansion}
if there exists a function $f:\bb{Z}_0^+\to\bb{Z}_0^+$ such that for every $G\in \GG$ and every non-negative integer $k$,
$\dd_k(G)\le f(k)$.  We say that $\GG$ is \emph{nowhere-dense} if there exists a function $h:\bb{Z}_0^+\to\bb{Z}_0^+$ such that for every $G\in \GG$ and every non-negative integer $k$,
$G$ does not contain a $(\le\!k)$-subdivision of $K_{h(k)}$ as a subgraph.

There are many equivalent definitions of nowhere-dense
graph classes and classes of graphs with bounded expansion~\cite{nesbook}.  For example, to clarify the relationship to bounded expansion
classes, it is also possible to define nowhere dense classes in the terms of the density of graphs whose bounded depth subdivisions
appear in the graphs of the class, as follows.
A function $g:\bb{Z}_0^+\to\bb{Z}_0^+$ is \emph{subpolynomial} if $\lim_{n\to\infty} \frac{\log g(n)}{\log n}=0$,
and a function $f:\bb{Z}_0^+\times \bb{Z}_0^+\to\bb{Z}_0^+$ is \emph{subpolynomial in the second argument} if
for every non-negative integer $k$, the function $f(k,\cdot)$ is subpolynomial.
A class $\GG$ is nowhere-dense if and only if
there exists a function $f:\bb{Z}_0^+\times \bb{Z}_0^+\to\bb{Z}_0^+$ subpolynomial in the second argument
such that for every non-negative integer $k$ and graph $G\in\GG$, each subgraph $G'$ of $G$ satisfies $\dd_k(G')\le f(k,|V(G')|)$.

The notions of bounded expansion and nowhere-denseness formalize in a robust way the concept of sparseness of a graph class.
Such graph classes have a number of important algorithmic and structural properties,
including fixed-parameter tractability of any problem expressible in first-order logic when restricted to
a class with bounded expansion~\cite{dkt-fol} or a nowhere-dense class~\cite{grohe2014deciding}
and existence of low tree-depth colorings~\cite{grad1}.
Also importantly, many naturally defined graph classes have bounded expansion, including
proper minor closed classes, classes with bounded maximum degree, and more generally proper classes
closed on topological minors, graphs drawn in a fixed surface with a bounded number of crossings on each edge,
all graph classes with strongly sublinear separators,
and many others.  We refer the reader to the book of Ne{\v{s}}et\v{r}il and Ossona de Mendez~\cite{nesbook}
for a more thorough treatment of the subject.

Most of the mentioned examples of graph classes with bounded expansion are closed on taking subgraphs.
Graph classes that are only closed on induced subgraphs, and in particular graph classes characterized
by forbidden induced minors or induced subdivisions, has been less studied in the context.  The major
issue is that such classes typically contain arbitrarily large cliques or bicliques (balanced complete bipartite
graphs $K_{s,s}$), which have unbounded minimum degree.  However, K\"uhn and Osthus~\cite{kuhndens}
shown that at least with regards to the maximum average degree $\dd_0$, this is the only obstruction.

\begin{theorem}[{K\"uhn and Osthus~\cite[Theorem 1]{kuhndens}}]\label{thm-kuhnost}
For every graph $H$ and a positive integer $s$, there exists an integer $d$ such that
every graph with average degree at least $d$ contains either $K_{s,s}$ as a subgraph or a subdivision of $H$ as an \emph{induced} subgraph.
\end{theorem}

Note that if $G$ contains a large biclique as a subgraph, applying Ramsey's theorem to each part of the biclique gives either
a large clique or a large biclique as an induced subgraph.  Hence, the previous theorem can be reformulated as follows.

\begin{corollary}\label{cor-kuhnost}
For every graph $H$ and a positive integer $s$, there exists an integer $d$ as follows.
If $\GG$ is a class of graphs that do not contain $K_s$, $K_{s,s}$, or any subdivision
of $H$ as an induced subgraph, then $\dd_0(G)\le d$ for all $G\in\GG$.
\end{corollary}

In this note, we show that the result of K\"uhn and Osthus~\cite{kuhndens} can be easily extended to
prove that such graph classes actually have bounded expansion.

\begin{theorem}\label{thm-main1}
For every graph $H$ and a positive integer $s$, if $\GG$ is a class of graphs that do not
contain $K_s$, $K_{s,s}$, or any subdivision of $H$ as an induced subgraph, then $\GG$ has bounded expansion.
\end{theorem}

Theorem~\ref{thm-main1} is a consequence of the following new characterization of graph classes with bounded expansion.  
For a graph $G$, let $\ddin_k(G)$ denote the maximum of average degrees of graphs $H$ whose $(\le\!k)$-subdivision appears as an \emph{induced} subgraph
in $G$.  The \emph{$k$-subdivision} of a graph $H$ is the graph obtained from $H$ by subdividing each edge exactly $k$ times.
Let $\ddinex_k(G)$ denote denote the maximum of average degrees of graphs $H$ whose $k$-subdivision appears as an induced subgraph
in $G$.

\begin{theorem}\label{thm-maindens}
For a class of graphs $\GG$, the following statements are equivalent.
\begin{itemize}
\item[(a)] $\GG$ has bounded expansion.
\item[(b)] There exists a function $f:\bb{Z}_0^+\to\bb{Z}_0^+$ such that for every $G\in \GG$ and every non-negative integer $k$,
$\ddin_k(G)\le f(k)$.
\item[(c)] There exists a function $f:\bb{Z}_0^+\to\bb{Z}_0^+$ such that for every $G\in \GG$ and every non-negative integer $k$,
$\ddinex_k(G)\le f(k)$.
\end{itemize}
\end{theorem}

The same argument gives a characterization of nowhere-dense graph classes.
\begin{theorem}\label{thm-maindens-nd}
For a class of graphs $\GG$, the following statements are equivalent.
\begin{itemize}
\item[(a)] $\GG$ is nowhere-dense.
\item[(b)] There exists a function $f:\bb{Z}_0^+\times \bb{Z}_0^+\to\bb{Z}_0^+$ subpolynomial in the second argument
such that for every non-negative integer $k$ and graph $G\in \GG$, every induced subgraph $G'$ of $G$ satisfies $\ddin_k(G')\le f(k,|V(G')|)$.
\item[(c)] There exists a function $f:\bb{Z}_0^+\times \bb{Z}_0^+\to\bb{Z}_0^+$ subpolynomial in the second argument
such that for every non-negative integer $k$ and graph $G\in \GG$, every induced subgraph $G'$ of $G$ satisfies
$\ddinex_k(G')\le f(k,|V(G')|)$.
\end{itemize}
\end{theorem}

\section{The proofs}

Let $G$ be a bipartite graph with bipartition $(A,B)$.  A \emph{hat} over this bipartition is a $3$-vertex path
in $G$ with endpoints in $B$ (and the midpoint in $A$).  We say that a set $\PP$ of hats is \emph{uncrowded}
if any two hats in $\PP$ join distinct pair of vertices and have distinct midpoints.  We say that it is \emph{induced}
if the midpoint of each hat has exactly two neighbors in $B$, i.e., $\bigcup\PP$ is an induced subgraph of $G$.

K\"uhn and Osthus~\cite{kuhndens} proved the following.
\begin{lemma}[{K\"uhn and Osthus~\cite[Lemma~18]{kuhndens}}]\label{lemma-inducehats}
Let $r$ be a positive integer and let $G$ be a bipartite graph with bipartition $(A,B)$, such that each vertex of $A$ has degree at most $4r$.
If $G$ contains an uncrowded set of at least $\tfrac{r^{11}}{2^8}|B|$ hats, then $G$ has an induced subgraph $G'$ with bipartition
$(A',B')=(A\cap V(G'),B\cap V(G'))$ such that $B'\neq\emptyset$ and the set of all $3$-vertex paths in $G'$ with midpoints in $A'$
forms an induced uncrowded set of at least $\tfrac{r^9}{2^{15}}|B'|$ hats over $(A',B')$.
\end{lemma}

We also need another lemma of K\"uhn and Osthus~\cite{kuhndens} (the \emph{branch vertices} of a subdivision
of a graph $H$ are the vertices of the subdivision corresponding to the original vertices of $H$).
\begin{lemma}[{K\"uhn and Osthus~\cite[Lemma~20]{kuhndens}}]\label{lemma-fixb}
Let $r\ge 2^{25}$ be an integer.
Let $(A,B)$ be a partition of vertices of a graph $G$ such that $A$ is an independent set of $G$, $\chi(G[B])\le r$
and $\dd_0(G[B])\le r^3$.  Let $G'$ be the spanning bipartite subgraph of $G$ containing exactly the edges of $G$
with one end in $A$ and the other end in $B$.  If $G'$ contains an induced uncrowded set of at least $\tfrac{r^9}{2^{15}}|B|$ hats
over $(A,B)$, then $G$ contains an induced subgraph $G''$ such that $G''$ is the $1$-subdivision of a graph of average degree at least $r$,
with all branch vertices contained in $B$.
\end{lemma}

Combining these lemmas with Theorem~\ref{thm-kuhnost} gives the following result.

\begin{lemma}\label{lemma-main1}
For all integers $r,k,s\ge 1$, there exists an integer $d_{r,k,s}=O((rsk)^{12})$ such that
for every graph $G$, if $\dd_0(G)\le s$ and $\ddinex_k(G)<r$,
then $\dd_k(G)<\dd_{k-1}(G)+d_{r,k,s}$.
\end{lemma}
\begin{proof}
Without loss of generality, we can assume that $r\ge \max(2^{25},s+1,sk/2)$.
Let $d_{r,k,s}=\tfrac{r^{11}(sk+1)}{2^6}$.

Suppose for a contradiction that $\dd_k(G)\ge \dd_{k-1}(G)+d_{r,k,s}$.  Let $H$ be a graph of average degree at least
$\dd_{k-1}(G)+d_{r,k,s}$ whose $(\le\!k)$-subdivision appears as a subgraph of $G$.  That is, there exists a function
$\varphi$ that assigns to vertices of $H$ distinct vertices of $G$ and to edges of $H$ paths of length at most $k+1$ in $G$,
such that for every $uv\in E(H)$,
the path $\varphi(uv)$ has endpoints $\varphi(u)$ and $\varphi(v)$ and its internal vertices do not belong to $\varphi(V(H))$, and for distinct edges $e_1,e_2\in E(H)$, the
paths $\varphi(e_1)$ and $\varphi(e_2)$ do not intersect except possibly in their endpoints.  
Without loss of generality, we can assume that $\varphi(e)$ is an induced path in $G$ for every $e\in E(H)$.
Let $B=\varphi(V(H))$.  Let $H_1$ be the subgraph of $H$ consisting of the edges such that the path $\varphi(e)$
has length exactly $k+1$.  Since $H$ has average degree at least $\dd_{k-1}(G)+d_{r,k,s}$ and a $(\le\!k-1)$-subdivision
of $H-E(H_1)$ is a subgraph of $G$, we conclude that the graph $H_1$ has average degree at least $d_{r,k,s}$.

Let $G_1$ be an auxiliary graph with vertex set $E(H_1)$, such that distinct $e_1,e_2\in E(H_1)$ are adjacent in $G_1$ if and
only if there exists an edge of $G$ with one end in an internal vertex of $\varphi(e_1)$ and the other end in an internal vertex of $\varphi(e_2)$.
Let $C$ be any subset of $E(H_1)=V(G_1)$, and let $D$ be the union of internal vertices of the paths $\varphi(e)$ for $e\in C$;
we have $|D|=k|C|$, and $|E(G_1[C])|\le |E(G[D])|\le \tfrac{\dd_0(G)}{2}|D|\le \tfrac{sk}{2}|C|$. Hence, $\dd_0(G_1)\le sk$,
and in particular $\chi(G_1)\le sk+1$.  Hence, $G_1$ contains an independent set $S$ of size at least $\tfrac{|E(H_1)|}{sk+1}$.
Let $H_2$ be the spanning subgraph of $H_1$ with edge set $S$, and note that the average degree of $H_2$ is at least $\tfrac{d_{r,k,s}}{sk+1}=\tfrac{r^{11}}{2^6}$.
By the choice of $S$, for any distinct $e_1,e_2\in E(H_2)$, there are no edges between the internal vertices of $\varphi(e_1)$ and $\varphi(e_2)$ in $G$.

Let $A_2=E(H_2)$ and let $G_2$ be an auxiliary bipartite graph with bipartition $(A_2,B)$, where $ev$ with $e\in E(H_2)$ and $v\in B$ is an edge of $G_2$ if and only if there
exists an edge of $G$ between an internal vertex of $\varphi(e)$ and $v$.  Let $D_2$ be the union of internal vertices of the paths
$\varphi(e)$ for $e\in A_2$.  Since the average degree of $H_2$ is greater than $2$, we have $|B|<|A_2|\le |D_2|$.
Note that $|E(G_2)|\le |E(G[D_2\cup B])|\le \tfrac{\dd_0(G)}{2}|D_2\cup B|\le s|D_2|\le sk|A_2|\le 2r|A_2|$.
Hence, less than half of the vertices of $A_2$ have degree more than $4r$ in $G_2$.  Let $A_2'\subseteq A_2$ consist of
vertices of $A_2$ whose degree in $G_2$ is at most $4r$, and let $G'_2$ be the induced subgraph of $G_2$
with bipartition $(A'_2,B)$.  We have $|A_2'|\ge |A_2|/2=|E(H_2)|/2\ge \tfrac{r^{11}}{2^8}|B|$.
Let $\PP$ be the set of paths $uev$ in $G_2$, where $e\in A_2'$ and $u$ and $v$ are the endpoints of $\varphi(e)$;
then $\PP$ is an uncrowded set of at least $\tfrac{r^{11}}{2^8}|B|$ hats over $(A'_2,B)$.  Let $G_3$ with bipartition $(A_3,B_3)$
be the induced subgraph of $G'_2$ obtained by applying Lemma~\ref{lemma-inducehats}, and let $H_3$ be the subgraph of
$H_2$ with vertex set $\varphi^{-1}(B_3)$ and edge set $A_3$.  We have $|E(H_3)|=|A_3|\ge \tfrac{r^9}{2^{15}}|B_3|$,
and for each $e\in E(H_3)$, the internal vertices of $\varphi(e)$ have no neighbors in $B_3$ other than the endpoints of $\varphi(e)$.

Finally, we consider an auxiliary graph $G_4=G_3\cup G[B_3]$.
Note that $\dd_0(G_3[B_3])=\dd_0(G[B_3])\le s\le r^3$ and $\chi(G_3[B_3])\le \dd_0(G[B_3])+1\le s+1\le r$.  
Let $G_5$ be the induced bipartite subgraph of $G_4$ obtained by applying Lemma~\ref{lemma-fixb}, with bipartition
$(A_5,B_5)=(A_3\cap V(G_5),B_3\cap V(G_5))$.  Then $\varphi$ restricted to $A_5\cup B_5$ shows that $G$ contains
the $k$-subdivision of a graph of average degree at least $r$ as an induced subgraph, and thus
$\ddinex_k(G)\ge r$; this is a contradiction.
\end{proof}

The new characterizations of classes with bounded expansion and nowhere-dense classes now readily follow.

\begin{proof}[Proof of Theorem~\ref{thm-maindens}]
For every graph $G$ and a non-negative integer $k$, we have $\dd_k(G)\ge\ddin_k(G)\ge\ddinex_k(G)$,
and thus $\text{(a)}\Rightarrow\text{(b)}\Rightarrow\text{(c)}$.

Suppose now that $f:\bb{Z}_0^+\to\bb{Z}_0^+$ is a function such that for every $G\in \GG$ and every non-negative integer $k$,
$\ddinex_k(G)\le f(k)$.
Let $g(0)=f(0)$ and for every positive integer $k$, let $d_{r,k,s}$ be the constant from Lemma~\ref{lemma-main1}, where $r=f(k)+1$ and $s=f(0)$,
and let $g(k)=g(k-1)+d_{r,k,s}$.
Consider a graph $G\in \GG$.  By induction, we will show that $\dd_k(G)\le g(k)$ for every non-negative integer $k$.
For $k=0$, we have $\dd_0(G)=\ddinex_0(G)\le f(0)=g(0)$; hence, we can assume that $k\ge 1$ and that
$\dd_{k-1}(G)\le g(k-1)$.  However, then Lemma~\ref{lemma-main1} implies
$\dd_k(G)<\dd_{k-1}(G)+d_{f(k)+1,k,f(0)}\le g(k-1)+d_{f(k)+1,k,f(0)}=g(k)$.
We conclude that $\GG$ has bounded expansion, and thus $\text{(c)}\Rightarrow\text{(a)}$.
\end{proof}

\begin{proof}[Proof of Theorem~\ref{thm-maindens-nd}]
Again, the implications $\text{(a)}\Rightarrow\text{(b)}\Rightarrow\text{(c)}$ are trivial.

Suppose now that $f:\bb{Z}_0^+\times \bb{Z}_0^+\to\bb{Z}_0^+$ is a function subpolynomial in the second argument
such that for every non-negative integer $k$ and graph $G\in\GG$,
every induced subgraph $G'$ of $G$ satisfies $\ddinex_k(G')\le f(k,|V(G')|)$.
We can without loss of generality assume that $f$ is non-decreasing in the second argument.
Let us define $g(0,n)=f(0,n)$.  For any positive integer $k$, let $d_{r(k,n),k,s(n)}$ be the constant from Lemma~\ref{lemma-main1},
where $r(k,n)=f(k,n)+1$ and $s(n)=f(0,n)$,
and let $g(k,n)=g(k-1,n)+d_{r(k,n),k,s(n)}$.  As in the proof of Theorem~\ref{thm-maindens}, induction on $k$ shows that
$\dd_k(G')\le g(k,|V(G')|)$ for every (induced) subgraph $G'$ of a graph $G\in \GG$.  Furthermore, for any fixed $k$,
$$g(k,n)=f(0,n)+\sum_{t=1}^k d_{r(t,n),t,s(n)}=f(0,n)+\sum_{t=1}^k O((f(t,n)f(0,n)t)^{12}),$$
and thus $g$ is subpolynomial in the second argument.  It follows that $\GG$ is nowhere-dense, and thus $\text{(c)}\Rightarrow\text{(a)}$.
\end{proof}

To derive Theorem~\ref{thm-main1}, we also need the following result.

\begin{theorem}[Koml{\'o}s and Szemer{\'e}di~\cite{KoSz}, Bollob\'as and Thomason~\cite{BoTh}]\label{thm-clique}
There exists $c>0$ such that for every positive integer $n$, every graph of average degree at least $cn^2$ contains
a subdivision of $K_n$ as a subgraph.
\end{theorem}

\begin{proof}[Proof of Theorem~\ref{thm-main1}]
Let $n=|V(H)|$, and let $c$ be the constant from Theorem~\ref{thm-clique}.  Let $d$ be the constant
of Corollary~\ref{cor-kuhnost} for $H$ and $s$.  Let $f(0)=d$ and let $f(k)=cn^2$ for every $k\ge 1$.

Consider any graph $G\in \GG$.  By assumptions and Corollary~\ref{cor-kuhnost}, every induced subgraph of $G$
has average degree at most $d$, and thus $\ddinex_0(G)\le f(0)$.
Consider any positive integer $k$ and let $H_1$ be any graph whose $k$-subdivision appears in $G$ as an induced subgraph $H'_1$.
Since $H'_1$ does not contain a subdivision of $H$ as an induced subgraph, and each edge of $H_1$ is subdivided at least once to obtain $H'_1$,
we conclude that $H_1$ does not contain a subdivision of $H$ as a subgraph.  Consequently, $H_1$ does not contain a subdivision of $K_n$ as
a subgraph, and the average degree of $H_1$ is less than $cn^2$ by Theorem~\ref{thm-clique}.  It follows that
$\ddinex_k(G)<cn^2=f(k)$ for every positive integer $k$.  Hence, $\GG$ has bounded expansion by Theorem~\ref{thm-maindens}.
\end{proof}

\section*{Acknowledgments}

I would like to thank Ken-ichi Kawarabayashi for fruitful discussions that motivated this result,
and in particular for pointing my attention to the result of K\"uhn and Osthus.

\bibliographystyle{siam}
\bibliography{subdivbexp}

\begin{thebibliography}{1}

\bibitem{BoTh}
{\sc B.~Bollob\'as and A.~Thomason}, {\em Proof of a conjecture of {M}ader,
  {E}rd{\H{o}}s and {H}ajnal on topological complete subgraphs}, European J.
  Combin., 19 (1998), pp.~883--887.

\bibitem{dkt-fol}
{\sc Z.~Dvo\v{r}\'ak, D.~Kr{\'a}l', and R.~Thomas}, {\em Deciding first-order
  properties for sparse graphs}, in FOCS, IEEE Computer Society, 2010,
  pp.~133--142.

\bibitem{grohe2014deciding}
{\sc M.~Grohe, S.~Kreutzer, and S.~Siebertz}, {\em Deciding first-order
  properties of nowhere dense graphs}, in Proceedings of the 46th Annual ACM
  Symposium on Theory of Computing, ACM, 2014, pp.~89--98.

\bibitem{KoSz}
{\sc J.~Koml{\'o}s and E.~Szemer{\'e}di}, {\em Topological cliques in graphs.
  {II}}, Combin. Probab. Comput., 5 (1996), pp.~79--90.

\bibitem{kuhndens}
{\sc D.~K{\"u}hn and D.~Osthus}, {\em Induced subdivisions in {$K_{s,s}$}-free
  graphs of large average degree}, Combinatorica, 24 (2004), pp.~287--304.

\bibitem{grad1}
{\sc J.~Ne{\v{s}}et\v{r}il and P.~{Ossona de Mendez}}, {\em Grad and classes
  with bounded expansion {I}. {D}ecomposition.}, European J. Combin., 29
  (2008), pp.~760--776.

\bibitem{nesbook}
\leavevmode\vrule height 2pt depth -1.6pt width 23pt, {\em Sparsity -- Graphs,
  Structures, and Algorithms}, Springer, 2012.

\end{thebibliography}

\end{document}